\documentclass[reqno, a4paper]{amsart}

\usepackage[english]{babel}
\usepackage[utf8x]{inputenc}
\usepackage[T1]{fontenc}
\usepackage{yfonts}
\usepackage{frcursive}
\usepackage{aurical}
\usepackage{egothic}
\usepackage{miama}
\usepackage{mathtools}
\usepackage{indentfirst}

\usepackage[a4paper,top=3cm,bottom=2cm,left=2.5cm,right=2.5cm,marginparwidth=2cm]{geometry}

\usepackage{amsmath, latexsym, graphicx, amsthm, amssymb, scrextend, setspace,wasysym, marvosym, tikz}
\usepackage[colorinlistoftodos]{todonotes}
\usepackage[colorlinks=true, allcolors=blue]{hyperref}
\usepackage{tikz-cd}

\newcommand{\F}{\mathcal{F}}
\newcommand{\LL}{\mathcal{L}}
\newcommand{\OO}{\mathcal{O}}
\newcommand{\PP}{\mathbb{P}}

\newcommand{\BN}{\mathcal{BN}}
\newcommand{\pic}{\textup{\textbf{Pic}}}
\DeclareMathOperator{\Pic}{Pic} 

\DeclareMathOperator{\spec}{Spec}
\DeclareMathOperator{\fitt}{Fitt}

\newcommand{\be}{\begin{equation}} 
\newcommand{\ee}{\end{equation}}   
\newcommand{\beq}{\begin{eqnarray*}} 
\newcommand{\eeq}{\end{eqnarray*}} 
\newcommand{\tr}{\textup}

\newcommand{\kdot}{{{\,\begin{picture}(1,1)(-1,-2)\circle*{2}\end{picture}\,}}}
\DeclareMathOperator{\va}{\mathcal{VA}}
\DeclareMathOperator{\bpf}{\mathcal{BF}}

\DeclareMathOperator{\sing}{Sing}

\makeatletter
\newcommand*\bigcdot{\mathpalette\bigcdot@{.5}}
\newcommand*\bigcdot@[2]{\mathbin{\vcenter{\hbox{\scalebox{#2}{$\m@th#1\bullet$}}}}}
\makeatother

\newtheorem{innercustomthm}{Theorem}
\newenvironment{mythm}[1]
  {\renewcommand\theinnercustomthm{#1}\innercustomthm}
  {\endinnercustomthm}

\newtheorem{innercustomcor}{Corollary}
\newenvironment{mycor}[1]
  {\renewcommand\theinnercustomcor{#1}\innercustomcor}
  {\endinnercustomcor}

\newtheorem{theorem}{Theorem}

\newtheorem{lemma}[theorem]{Lemma}
\newtheorem{definition}[theorem]{Definition}

\newtheorem{corollary}[theorem]{Corollary}
\newtheorem{proposition}[theorem]{Proposition}

\newtheorem{question}[theorem]{Question}
\numberwithin{theorem}{section}
\numberwithin{equation}{section}
\theoremstyle{definition}
\newtheorem{example}[theorem]{Example}
\newtheorem{remark}[theorem]{Remark}

\theoremstyle{remark}

\usepackage{pdfpages}

\title{Moduli of Very Ample Line Bundles}
\author{Brian Nugent}
\date{\today}

\begin{document}

\begin{abstract}
    Let $X$ be a projective variety over a field. In this paper, we will construct a moduli space of very ample line bundles on $X$. In doing so, we develop a generalization of Fitting ideals to complexes of sheaves on $X$. We give other applications of these Fitting ideals such as constructing Brill-Noether spaces for higher dimensional varieties and giving a scheme structure to the locus where the projective dimension of a module jumps up. 
\end{abstract}

\maketitle

\setcounter{tocdepth}{1}

{\setlength{\parskip}{0pt}

\tableofcontents
}

\section{Introduction}

Let $X$ be a projective variety over a field. A very ample line bundle $\LL$ on $X$ is a line bundle where the associated rational map $\phi_\LL: X \rightarrow \mathbb{P}(H^0(X,\LL))$ is an embedding. In this way, all embeddings of $X$ into projective space are determined by a choice of very ample line bundle (and possibly a projection to a hyperplane). Very ample line bundles are of great importance in algebraic geometry as they allow one to use the benefits of having an embedding into projective space without having to worry about working in coordinates. In this paper, we construct a moduli space for very ample line bundles on $X$. 

To do this, we need to first define what families our moduli space should parameterize. In section \ref{examples} we show that the naive definition for the moduli functor does not produce a moduli space. In section \ref{liftable-sections}, we define what it means for a family of line bundles to have \textit{liftable sections} and in section \ref{deforming-embeddings}, we explain why this is a natural condition to put on families of very ample line bundles. Put imprecisely, a family of line bundles $L$ over $T$ has liftable sections if the associated rational maps $\phi_{L_t}$ deform in a family over $T$ as well. Our main result is the following:

\begin{mythm}{\ref{va-moduli}}
            Let $X$ be a projective variety. There exists a scheme $\va_X$ representing the functor
$$ T \rightarrow  \left\{ L \in \Pic_X(T) : \forall t \in T, L_t \tr{ is very ample and } L \tr{ has liftable sections over } T \right\} $$
Moreover, the natural map
$$ \va_X \rightarrow \pic_X $$
is a locally closed embedding on the connected components of $\va_X$.

\end{mythm}

\iffalse

\begin{mythm}{\ref{BN-ideals}}
    Let $f: X \rightarrow Y$ be a proper flat morphism of locally noetherian schemes and $\F$ a coherent sheaf on $X$, flat over $Y$. Then there exist natural determinantal subschemes $W^k(\F) \subseteq Y$ such that
    $$ \tr{Supp } W^k(\F) = \{y \in Y | h^0(X_y,\F_y) \geq k+1 \} $$

\end{mythm}

From the theory of determinantal schemes, we get that
$$ W^{k+1}(\F) \subseteq \tr{Sing}(W^k(\F)) $$

whenever $W^k(\F) \neq Y$.

As a consequence of Theorem \ref{BN-ideals}, we get a direct generalization of the classical Brill-Noether varieties.

\fi

The main technical tool in our proof is a generalization of Fitting ideals for chain complexes. Given a bounded above complex $\mathcal{E}^\kdot$ of locally free sheaves on $X$, for each $i,k \in \mathbb{Z}$ we define the Fitting ideal of $\mathcal{E}^\kdot$, $\fitt_k^i(\mathcal{E}^\kdot)$, to be the ideal locally generated by the $(\chi_i-k) \times (\chi_i-k)$ minors of $d^i$ where $\chi_i$ is the alternating sum of the ranks of the $\mathcal{E}^j$ for $j > i$. For a more precise definition, see Section \ref{higher_fitt}. The crucial fact about these Fitting ideals is the following,

\begin{mythm}{\ref{Fitting-ideal}}
    Let $\mathcal{E}^\kdot \rightarrow \mathcal{W}^\kdot$ be a quasi-isomorphism of bounded above complexes of locally free sheaves on $X$. Then $\fitt_k^i(\mathcal{E}^\kdot) = \fitt_k^i(\mathcal{W}^\kdot)$ for all $i,k \in \mathbb{Z}$.
\end{mythm}

This allows us to define Fitting ideals for any element of $D^-_{coh}(X)$ (the bounded above derived category with coherent cohomology) as the Fitting ideals of a locally free resolution. This generalization of Fitting ideals is useful in defining natural scheme structures on certain closed subsets of $X$. One notable example is the following:

\begin{mycor}{\ref{projective-dimension}}
        Let $R$ be a commutative noetherian ring and $M$ an $R$-module of rank $k$. The Fitting ideal $\fitt_{(-1)^{d-1}k}^{d-1}(M)$ defines a natural determinantal scheme structure on the set
    $$ \{ \mathfrak{p} \in \spec R | \tr{ pd}(M_\mathfrak{p}) \geq d \} $$
\end{mycor}

Our main application of these Fitting ideals is in constructing a generalization of classical Brill-Noether varieties which in turn we use to prove Theorem \ref{va-moduli}.

Let $C$ be a smooth projective curve of genus $g$ over an algebraically closed field. The Brill-Noether varieties are the locus $W_d^r \subseteq \pic_C^d$ of line bundles of degree $d$ whose space of global sections has dimension at least $r+1$ inside the Picard scheme of $C$. These spaces come with natural scheme structures derived as degeneracy loci of a certain map of vector bundles. See \cite{ACGH1} for more details on this beautiful subject.

For higher dimensional $X$, semicontinuity tells us that $\{\LL \in \pic_X : h^0(X,\LL) \geq k+1\}$ is closed. For our application to proving Theorem \ref{va-moduli}, it will be crucial that we give this closed subset the correct scheme structure. In recent years, there has been some success in generalizing Brill-Noether theory to higher dimensional varieties and to more general moduli spaces of coherent sheaves, see \cite{coskun2023brillnoether} for a survey of these results on surfaces and \cite{Costa10} for some results in any dimension. In addition to many other results, these papers give natural scheme structures to the Brill-Noether loci similarly to how they are done for curves, which requires one to assume that $H^i(X,\F) = 0$ for all $i \geq 2$ and all $\F$ in the given moduli space. We introduce a new construction, which generalizes the existing ones and does not require the assumption on cohomology vanishing. Our main result is the following:

\begin{mycor}{\ref{Brill-Picard}}
    Let $X$ be a projective variety and let $\pic_X$ be the Picard scheme of $X$ (see \ref{picard-scheme}). Then there exist closed subschemes $W^k_X \subseteq \pic_X$ such that
    $$ \tr{Supp } W^k_X = \{\LL \in \pic_X : h^0(X,\LL) \geq k+1\} $$
    and 
    $$ \dot{\bigcup}_{r \in \mathbb{Z}} W_X^r \setminus W_X^{r+1} \tr{ represents the functor } $$
$$ T \rightarrow  \left\{ L \in \Pic_X(T) : L \tr{ has liftable sections over } T \right\} $$
\end{mycor}

We prove a more general version that can be applied to any moduli space of coherent sheaves in Theorems \ref{BN-ideals} and \ref{functorial}.

Just like in the curve case, there is a map of vector bundles whose degeneracy loci give us the $W^k_X$. This can be used to prove relations between them, for example:
$$ W_X^{k+1} \subseteq \sing (W_X^k) $$

See \ref{sing-determinantal-subschemes}.

In Corollary \ref{classical}, we show that our construction agrees with the classical Brill-Noether varieties in the case $X$ is a smooth curve. We also show that our construction agrees with the Brill-Noether loci defined in \cite{Costa10} and \cite{coskun2023brillnoether} in the case where they are both defined.

In Section \ref{fitting ideals-chapter}, we review the theory of Fitting ideals. In Section \ref{higher fitting ideals-chapter}, we define Fitting ideals for objects of $D_{coh}^-(X)$ and give some applications of them. In Section \ref{Brill-Noether Ideals} we give a general construction for scheme-theoretic Brill-Noether loci that applies in particular to any moduli of coherent sheaves. In Section \ref{liftable-sections} we show that these Brill-Noether loci produce a stratification with the universal property that parameterizes families of sheaves with liftable sections. In Section \ref{Brill-Noether Loci-chapter}, we apply the construction from the previous section to the Picard scheme. In Section \ref{Moduli of VA-chapter}, we prove our main result, constructing a moduli space of very ample line bundles. In the brief Section \ref{moduli of BF-chapter} we describe the small changes to the argument required to construct a moduli space for basepoint free line bundles.

\subsection{Acknowledgments}
This paper is based on my PhD thesis. I would like to thank my advisor Sándor Kovács for all of his help throughout the writing of this. I would also like to thank Jarod Alper, Max Lieblich and Giovanni Inchiostro for helpful conversations and suggestions.

\section{Preliminaries} \label{preliminaries}

\subsection{Notation and Conventions}

Throughout this paper, all schemes are defined over a fixed algebraically closed field $\kappa$. Most of our results generalize to all fields but require adjustments to the assumptions so for the sake of simplicity, we leave this to the interested reader. A variety will mean an integral scheme of finite type over $\kappa$.

Given morphisms $X \rightarrow Y$ and $T \rightarrow Y$, we denote $X \underset{Y}{\times} T$ by $X_T$. If $\F$ is a sheaf on $X$ and $g: X_T \rightarrow X$ is the projection, then we denote $g^*\F$ by $\F_T$. In particular, if $y \in Y$ is a point then the fiber of $X \rightarrow Y$ at $y$ is denoted $X_y$ and $\F$ restricted to $X_y$ is denoted $\F_y$. We denote the projection from $X_T$ to $T$ as $\pi_T$.

We denote $\dim H^i(X,\F)$ by $h^i(X,\F)$.

Given a scheme $X$, we denote its Picard functor as $\Pic_X$ and its Picard scheme as $\pic_X$ (see section \ref{picard-scheme}). 

When working with complexes of sheaves we use the standard cochain complex notation of 
$$ \F_0 \rightarrow \F_1 \rightarrow \F_2 $$
but in section \ref{higher_fitt}, when we are working with complexes of modules over a ring, we use chain complex notation,
$$ M_2 \rightarrow M_1 \rightarrow M_0 $$

\subsection{The Picard Scheme} \label{picard-scheme}

The Picard group of a scheme $X$ is the group of all line bundles (or invertible sheaves) on $X$. Like many objects in algebraic geometry, it can be given the structure of a scheme itself (under mild conditions) where it parameterizes families of line bundles. In this section, we review the definition and basic properties of the Picard scheme. For more information and proofs of these properties, see \cite[\href{https://stacks.math.columbia.edu/tag/0B9R}{Tag 0B9R}]{stacks-project2} for the case of a curve or \cite{kleiman2005picard} for the general case.

\begin{definition}
    We define the Picard functor $\Pic_X: Sch/\kappa \rightarrow Set$ by
    $$ \Pic_X(T) = \Pic(X \times T)/\pi_T^* \Pic(T) $$
\end{definition}

In general, to obtain a representable functor, one should take the sheafification of this functor in the étale topology but in the case that $X$ has a $\kappa$-point (trivial in our case since $\kappa$ is algebraically closed) it turns out that $\Pic_X$ is already representable.

\begin{theorem}[{\cite[4.8]{kleiman2005picard}}]
Let $X$ be a projective variety. Then $\Pic_X$ is represented by a locally noetherian scheme $\pic_X$, whose connected components are quasi-projective.
\end{theorem}

By plugging $X$ into the definition of the $\Pic_X$ we obtain the following,

\begin{proposition}
    There exists a (non-unique) line bundle $\mathcal{U}$ on $X \times \pic_X$ such that for any scheme $T$ along with a line bundle $\mathcal{L}$ on $X \times T$, there exists a morphism $f: T \rightarrow \pic_X$ such that
    $$ \mathcal{L} \cong (1 \times f)^* \mathcal{U} \otimes \pi_T^* \mathcal{N}  $$
    where $\mathcal{N}$ is some line bundle on $T$.
\end{proposition}

We call $\mathcal{U}$ a universal line bundle. We can see from the definition of the Picard functor that a universal line bundle is unique up to tensor with the pullback of a line bundle on $\pic_X$.

\subsection{The Derived Category}
Let $X$ be a scheme. Let $D(X)$ be the derived category of $X$. For definitions and basic properties of the derived category, see \cite[\href{https://stacks.math.columbia.edu/tag/05QI}{Tag 05QI}]{stacks-project2}.

We let $D^b(X)$ (resp. $D^-(X)$) denote the derived category of bounded (resp. bounded above) complexes. If $R$ is a commutative ring, we let $D(R) = D(\spec R)$ and define $D^b(R)$ and $D^-(R)$ similarly. 

We let $D_{coh}(X) \subseteq D(X)$ be the subcategory of objects whose cohomology sheaves are coherent. We define $D^-_{coh}(X)$ and $D^b_{coh}(X)$ similarly.

Given a left (resp. right) exact covariant functor $F: Mod(\OO_X) \rightarrow Mod(\OO_Y)$, we denote the right (resp. left) derived functor of $F$ as $RF$ (resp. $LF$). Given $K \in D(X)$, we denote $R^iF(K) = h^i(RF(K))$.

\subsection{Moduli of Sheaves}

Let $\mathcal{C}$ be a class of coherent sheaves on a variety $X$. To simplify things, we will define a \textbf{moduli space of sheaves} for $\mathcal{C}$ to be a locally noetherian scheme $M$ along with a coherent sheaf $\mathcal{U}$ on $X \times M$ that is flat over $M$ such that for each $\F \in \mathcal{C}$, there exists a unique closed point $y \in M$ such that $\mathcal{U}_y \cong \F$. We call $\mathcal{U}$ the \textbf{universal sheaf}. Because of the condition on $\mathcal{U}$, we can refer to $\F$ as a point of $M$ without confusion. Usually one would require $M$ and $\mathcal{U}$ to satisfy some universal property in some sense that applies to the moduli problem at hand, but we have no need for this here. For a proper treatment of moduli of sheaves, see the book \cite{Huybrechts_Lehn_2010} by Huybrechts and Lehn.

The main example of a moduli of sheaves to keep in mind is the Picard scheme of a projective variety, see Proposition \ref{picard-scheme}. 

Now let us look at the case of degree $d$ line bundles on a smooth curve $C$. The moduli space in this case is $\pic^d_C$ and we have a universal line bundle $L$ on $C \times \pic^d_C$. Classical Brill-Noether theory uses $L$ to construct a natural determinantal scheme structure on the set 
$$ W_d^k = \{\LL \in \pic_C^d : h^0(C,\LL) \geq k + 1\} $$

In Corollary \ref{BN-moduli-of-sheaves} we generalize this to any moduli space of sheaves.

\section{Review of Determinantal Schemes and Fitting Ideals} \label{fitting ideals-chapter}

We fix a commutative noetherian ring $R$ for this section and the next.

\subsection{Ideals of Minors and Determinantal Schemes}
The results of this section and the next all rely on ideals of minors of maps of vector bundles so we will review their basic properties here.

Let $R^{\oplus n} \overset{\phi}{\rightarrow} R^{\oplus m}$, we can represent $\phi$ as a $n$ by $m$ matrix with entries in $R$. We denote the ideal of $R$ that is generated by the $k \times k$ minors of $\phi$ as $I_k(\phi)$.

Let $X$ be a scheme. Given a map of vector bundles on $X$, $\mathcal{E} \overset{\psi}{\rightarrow} \mathcal{W}$, we can glue together the ideals of minors that we get on local trivializations of $\mathcal{E}$ and $\mathcal{W}$ to obtain a sheaf of ideals $I_k(\psi)$. We let $X_k(\psi)$ denote the scheme defined by $I_{k+1}(\psi)$.

\begin{proposition}
    Let $X$ be a scheme and $\mathcal{E} \overset{\psi}{\rightarrow} \mathcal{W}$ a map of vector bundles on $X$. Then
    $$ \tr{Supp } X_k(\psi) = \{ p \in X : \tr{rk}(\psi_p) \leq k\} $$
\end{proposition}

\begin{proof}
    Let $p \in X$. Let $\psi$ be locally represented by the matrix $(r_{ij})$. Then the map $\psi_p: \mathcal{E} \otimes k(p) \rightarrow \mathcal{W} \otimes k(p)$ is represented by the matrix $(r_{ij}(p))$. The map $\psi_p$ has rank $\leq k$ if and only if all the $(k+1) \times (k+1)$ minors of $(r_{ij}(p))$ vanish which happens if and only if $p \in I_{k+1}(\psi)$.
\end{proof}

\begin{proposition} \label{minors-basechange}
   Let $f: Y \rightarrow X$ be a morphism of schemes. Let $\mathcal{E} \overset{\psi}{\rightarrow} \mathcal{W}$ be a morphism of vector bundles on $X$. Then $f^{-1}I_k(\psi) = I_k(f^* \psi)$.
\end{proposition}

\begin{proof}
   If $\psi$ is locally represented by the matrix $(r_{ij})$ then $f^* \psi$ is locally represented by $(f^* r_{ij})$.
\end{proof}

The following is well known to experts but we give a proof for convenience.
\begin{proposition} \label{sing-determinantal-subschemes}
    Let $\mathcal{E} \overset{\psi}{\rightarrow} \mathcal{W}$ be a map of vector bundles on a variety $X$. Suppose that $X_{k+1}(\psi) \neq X$. Then
    $$ X_k(\psi) \subseteq \tr{Sing } X_{k+1}(\psi) $$
\end{proposition}

\begin{proof}
    Let $n$ be the rank of $\mathcal{E}$ and $m$ be the rank of $\mathcal{W}$. Let $M = \mathbb{A}^{nm}$ be the space of $n \times m$ matrices. Working locally, we may assume both vector bundles are trivial and we obtain a map
    $$ \lambda: X \rightarrow M $$
    that maps a point $p$ to the matrix $\psi_p$. The $k^{th}$ \textit{generic determinantal variety}, $M_k$, is defined by the vanishing of the $(k+1) \times (k+1)$ minors in the coordinates of $M$. We see that $X_k(\phi)$ is the scheme theoretic preimage of $M_k$ under $\lambda$.

    It follows from the discussion in \cite[II.2]{ACGH1} that given a point $A \in M_k$, the tangent space of $M_{k+1}$ is $nm-$dimensional. In other words,
    $$ \mathcal{T}_A(M_{k+1}) = \mathcal{T}_A(M) $$

    Translating this back to $X$, we have that given a point $x \in X_k(\phi)$,
    $$ \mathcal{T}_x(X_{k+1}(\phi)) = \mathcal{T}_x(X) $$

    So if $X_{k+1}(\phi) \neq X$, then $x$ is singular in $X_{k+1}(\phi)$.
\end{proof}

Another fact that we obtain through this same method is that if $X_k(\phi) \neq \varnothing$ then it has codimension at most $(n-k)(m-k)$.

See \cite[II]{ACGH1} for more details on determinantal subschemes.

\subsection{Fitting Ideals}

First let us review the classical theory of Fitting ideals originally introduced in \cite{fitting1936}. Let $M$ be an finitely generated $R$-module and let 
$$ F \overset{\phi}{\rightarrow} R^{\oplus r} \rightarrow M \rightarrow 0 $$
be a free presentation of $M$. We define the $k^{th}$ Fitting ideal, $\fitt_k(M)$, to be $I_{r-k}(\phi)$. That is, $\fitt_k(M)$ is the ideal of $(n-k) \times (n-k)$ minors of $\phi$. So the Fitting ideals form an increasing sequence,
$$ 0 = \fitt_{-1}(M) \subseteq \fitt_{0}(M) \subseteq \fitt_1(M) \subseteq \dots \fitt_r(M) = R $$
where $r$ is the number of generators for $M$ used in the free presentation. Note that $\fitt_k(M)$ can equal $R$ even if $M$ cannot be generated by $k$ elements. The crucial fact about Fitting ideals is that they are independent of the choice of free resolution.

\begin{theorem}[{\cite[\href{https://stacks.math.columbia.edu/tag/07Z6}{Section 07Z6}]{stacks-project2}}] \label{fitting}
    $\fitt_k(M)$ is independent of the choice of free resolution of $M$.
\end{theorem}

Fitting ideals commute with base change since the pullback of a free presentation is a free presentation, therefore we can glue together fitting ideals for an arbitrary coherent sheaf $\F$ on a scheme $X$. More precisely, $\fitt_k(\F)$ is the sheaf of ideals where
$$ \Gamma(U,\fitt_k(\F)) = \fitt_k(\Gamma(U,\F)) $$
for any open affine $U \subseteq X$. One useful property of Fitting ideals is that they characterize when a sheaf is locally free.

\begin{proposition}[{\cite[\href{https://stacks.math.columbia.edu/tag/07ZD}{Tag 07ZD}]{stacks-project2}}] \label{locally-free-fitting}
    A coherent sheaf $\F$ on a noetherian scheme $X$ is locally free of rank $k$ if and only if $\fitt_k(\F) = \OO_X$ and $\fitt_{k-1}(\F) = 0$.
\end{proposition}

Fitting ideals are useful for giving a natural scheme structure to certain closed subsets of schemes. For example, given a finite type scheme $X$ of pure dimension $n$, we can give the singular locus of $X$ a natural scheme structure as follows. Consider a local embedding of $X$ into $\mathbb{A}^n$ where the ideal of $X$ is generated by $f_1,\dots,f_m$. The Jacobian criterion for smoothness says that $X$ is singular at a point $x$ if the $n-d \times n-d$ minors of the Jacobian matrix,

\[
\begin{bmatrix}
    \frac{\partial f_1}{\partial x_1}     & \dots & \frac{\partial f_1}{\partial x_n} \\
    \vdots & \ddots & \vdots \\
    \frac{\partial f_m}{\partial x_1}       & \dots & \frac{\partial f_m}{\partial x_n}
\end{bmatrix}
\]
all vanish at $x$. So it is natural to define $Sing(X)$ as the closed subscheme of $X$ defined by these minors. A priori, this depends on the embedding but we can use Fitting ideals to show that the scheme structure is independent of the embedding.

\begin{proposition}
Let $X$ be a scheme of finite type and pure dimension $d$. Then the subscheme defined by $\fitt_d(\Omega_X)$ has support equal to the non-smooth locus of $X$.
\end{proposition}

\begin{proof}
Recall that for any point $p \in X$, $\dim \Omega_X \otimes k(p) \geq d$ with equality if and only if $X$ is smooth at $p$. We take a presentations of $\Omega_X$,
$$ \mathcal{E}^1 \overset{\phi}{\rightarrow} \mathcal{E}^0 \rightarrow \Omega_X \rightarrow 0 $$
with $r = \tr{rk } \mathcal{E}^0$. Then we see that the locus where $\dim \Omega_X \otimes k(p) \geq d$ is the locus where the rank of $\phi$ is less than $r-d$ which is exactly the locus cut out by $\fitt_d(\Omega_X)$.
\end{proof}

We define the \textit{singular set of $X$}, $Sing(X)$ to be the subscheme defined by $\fitt_d(\Omega_X)$. Note that given a local embedding of $X$ into $\mathbb{A}^n$, we get a finite presentation for $\Omega_X$
$$ \OO_X^{\oplus m} \overset{J}{\rightarrow} {\Omega_{\mathbb{A}^n}}_{|X} \rightarrow \Omega_X \rightarrow 0 $$
where $J$ is the Jacobian matrix. So our new definition for $Sing(X)$ agrees with our intuitive one in the discussion above.

\begin{example}
Consider the quadric cone $z^2 - xy = 0$ in $\mathbb{A}^3$. The Jacobian is 

\[
\begin{bmatrix}
    -y     & -x & 2z \\
\end{bmatrix}
\]

So the singular locus will be $x=y=z=0$, the reduced origin. Now consider the cusp $y^2 - x^3 = 0$ in $\mathbb{A}^2$. Its Jacobian is

\[
\begin{bmatrix}
    -3x^2     & 2y \\
\end{bmatrix}
\]

So the singular locus will be the non-reduced scheme whose ideal is $(x^2,y)$. As we can see, the scheme structure of the singular locus can give information about the singularity is even when it is just a single point.

\end{example}

\section{Higher Fitting Ideals} \label{higher fitting ideals-chapter}

In this section we generalize the notion of Fitting ideals to an object of the derived category $D_{coh}^-(X)$.

\subsection{Definition and Basic Properties of Higher Fitting Ideals} \label{higher_fitt}

In this section, we define Fitting ideals for a bounded above complex of $R$-modules. See \cite{MR3751292} for a different way of generalizing Fitting ideals to complexes which they call cohomology jump ideals. 

Given a bounded above complex $K$, we can always take a free resolution, that is, we can find a bounded above complex of free $R$-modules $F$ and a quasi-isomorphism $F \rightarrow K$. The main result of this section is that the ideals of minors of the maps in $F$ depend only on $K$. That is, if we took a different free resolution we would get the same ideals of minors. In the case where $K$ is a single module concentrated in degree 0 and $F$ is a free resolution of $K$, ideals of minors of the first map in $F$ are exactly the classical Fitting ideals of $K$. Just like in the classical case, we need to keep track of the ranks of the modules in $F$ to obtain a precise statement.

\textbf{Notation:} In this section, we use chain complex notation instead of cochain complex notation which is used in the rest of the paper. This is because it is much cleaner when working with bounded above objects and projective resolutions. 

Now consider a bounded above complex of free $R$-modules

\begin{center}
\begin{tikzcd}
\dots \arrow[r, "d_{N+3}"]         &  F_{N+2} \arrow[r, "d_{N+2}"]  & F_{N+1} \arrow[r, "d_{N+1}"]  & F_{N} \arrow[r] & 0 
\end{tikzcd}
\end{center}
Define
$$ r_i = \tr{rank } F_i $$
$$ \chi_i = \sum_{j= N}^{i-1} (-1)^{i-j-1} r_j = r_{i-1}-r_{i-2}+\dots+ (-1)^{i-N-1} r_N$$

Note that if $F_\kdot$ is an exact complex of vector spaces then $\chi_i = \tr{rank } d_i$ so we may think of $\chi_i$ as the "expected rank" of $d_i$. Now we make our main definition.

\begin{definition}
Define the \textbf{Fitting ideals} of $F_\kdot$ to be
$$\fitt_k^i(F_\kdot) = I_{\chi_i - k}(d_i)$$

We define these for bounded above cochain complexes the same way with $\fitt_k^i(F_{\kdot}) = I_{\chi_i - k}(d^i)$, where $\chi_i$ is the alternating sum of the ranks of the modules above $F^i$.
\end{definition}

Geometrically, these ideals cut out the points where the rank of $d_i$ drops below its expected rank by at least $k$. Our main result in this section is that Fitting ideals are invariant under quasi-isomorphisms. 

\begin{theorem} \label{Fitting-ideal}
    Let $a: F_\kdot \rightarrow \widetilde{F}_{\cdot}$ be a quasi-isomorphism of bounded above complexes of free $R$-modules. Then $\fitt_k^i(F_\kdot) = \fitt_k^i(\widetilde{F}_\kdot)$ for all $i$ and $k$.
\end{theorem}

For convenience, let us assume $F_\kdot$ and $\widetilde{F}_\kdot$ both vanish past degree zero. So we have a quasi-isomorphism,

\begin{center}
\begin{tikzcd}
\dots \arrow[r, "d_3"]         & F_2 \arrow[r, "d_2"] \arrow[d, "a"]  & F_1 \arrow[r, "d_1"] \arrow[d, "a"]  & F_0 \arrow[r] \arrow[d, "a"] & 0 \\
\dots \arrow[r, "\widetilde{d}_3"] & \widetilde{F}_2 \arrow[r, "\widetilde{d}_2"] & \widetilde{F}_1 \arrow[r, "\widetilde{d}_1"] & \widetilde{F}_0 \arrow[r]        & 0
\end{tikzcd}
\end{center}

To simplify notation, we will use 1 and 0 to denote the identity and zero maps respectively. We define
$$ B_i = \tr{im } d_{i+1} $$
$$ C_i = F_i / B_i $$
and $\widetilde{B_i}$, $\widetilde{C_i}$ for $\widetilde{F}_\kdot$ similarly.

Since all of the modules are free, $a$ has a homotopy inverse, $b: \widetilde{F}_\kdot \rightarrow F_\kdot$ (see \cite[I, Lemma 4.5]{MR0222093}). This means in particular that there is a map $s: \widetilde{F}_0 \rightarrow \widetilde{F}_1$ such that $\widetilde{d}_1 s = 1-ab$.

\begin{lemma} \label{reduct}

With the notation as above, the following morphism of complexes is a quasi-isomorphism
\begin{center}
\begin{tikzcd}
\dots \arrow[r, "d_4"]         & F_3 \arrow[r, "d_3"] \arrow[d, "a"]  & F_2 \arrow[r, "d_2 \oplus 0"] \arrow[d, "a"]  & F_1 \oplus \widetilde{F}_0 \arrow[r] \arrow[d, "A"] & 0 \\
\dots \arrow[r, "\widetilde{d}_4"] & \widetilde{F}_3 \arrow[r, "\widetilde{d}_3"] & \widetilde{F}_2 \arrow[r, "0 \oplus \widetilde{d}_2"] & F_0 \oplus \widetilde{F}_1 \arrow[r, ]                & 0
\end{tikzcd}
\end{center}
where $A = \begin{pmatrix}
    d_1 & -b \\
    a & s
\end{pmatrix}$
\end{lemma}

\begin{proof}

    Since $a$ is a quasi-isomorphism, we only need to check that

    \begin{center}
        \begin{tikzcd}
        C_1 \oplus \widetilde{F}_0 \arrow[r, "A"] & F_0 \oplus \widetilde{C}_1
        \end{tikzcd}
    \end{center}
    is an isomorphism.

    Define a map $B : B_0 \oplus \widetilde{F}_0 \rightarrow F_0 \oplus \widetilde{B}_0$ by the matrix,
    $$ B = \begin{pmatrix}
    1 & -b \\
    a & 1-ab
\end{pmatrix} $$

Now one can check that the following diagram of short exact sequences commutes

\begin{center}
\begin{tikzcd}
0 \arrow[r] & h^1(F_\kdot) \arrow[d, "a"] \arrow[r] & C_1 \oplus \widetilde{F}_0 \arrow[r, "d_1 \oplus 1"] \arrow[d, "A"] & B_0 \oplus \widetilde{F}_0 \arrow[r] \arrow[d, "B"] & 0 \\
0 \arrow[r] & h^1(\widetilde{F}_\kdot) \arrow[r]        & F_0 \oplus \widetilde{C}_1 \arrow[r, "1 \oplus \widetilde{d}_1"]        & F_0 \oplus \widetilde{B}_0 \arrow[r]                & 0
\end{tikzcd}
\end{center}

The left map is an isomorphism since $a$ is assumed to be a quasi-isomorphism and $B$ is an isomorphism because it has the inverse
    $$ B^{-1} = \begin{pmatrix}
    1 - ba & b \\
    -a & 1
\end{pmatrix} $$

Therefore $A$ is an isomorphism by the 5-lemma.
\end{proof}

\begin{lemma} \label{fitt}
    With the notation as above, $\fitt_k^1(F_\kdot) = \fitt_k^1(\widetilde{F}_\kdot)$.
\end{lemma}

\begin{proof}
We have exact sequences

\begin{center}
\begin{tikzcd}
F_1 \arrow[r, "d_1"] \arrow[d] & F_0 \arrow[d] \arrow[r] & h^0(F_\kdot) \arrow[d] \arrow[r] & 0 \\
\widetilde{F}_1 \arrow[r, "\widetilde{d}_1"]   & \widetilde{F}_0 \arrow[r]   & h^0(\widetilde{F}_\kdot) \arrow[r]   & 0
\end{tikzcd}
\end{center}
where the right vertical arrow is an isomorphism. Therefore, $\fitt_k^1(F_\kdot) = \fitt_k(h^0(F_\kdot)) = \fitt_k(h^0(\widetilde{F}_\kdot)) = \fitt_k^1(F_\kdot)$ where the middle equality holds by Proposition \ref{fitting}.
\end{proof}

We can combine these two lemmas to prove our theorem.

\begin{proof}[Proof of Theorem \ref{Fitting-ideal}]
    The $i=1$ case is Lemma \ref{fitt}. It is easy to see that $I_k(d_2 \oplus 0) = I_k(d_2)$, therefore by applying Lemma \ref{fitt} to the complex obtained from Lemma \ref{reduct} we get that
    $$ I_{r_1+\widetilde{r}_0-k}(d_2) = I_{\widetilde{r}_1+r_0-k}(\widetilde{d}_2) $$

    Since this holds for all $k$, we can replace $k$ with $k - (r_0 + \widetilde{r}_0)$ on both sides and obtain,
    $$ I_{r_1-r_0-k}(d_2) = I_{\widetilde{r}_1-\widetilde{r}_0-k}(\widetilde{d}_2) $$

    This is exactly the $i=2$ case. We can now apply Lemma \ref{reduct} again for the $i=3$ case and so on. 
\end{proof}

Given a bounded complex of free $R$-modules, it is sometimes useful for our alternating sum of ranks to come from the other direction. Because of this, we define a variation as follows,

\begin{definition}
    Let $F^\kdot$ be a bounded cochain complex of free $R$-modules with notation as above. Assume $F^i = 0$ for $i < N$. Let
    $$ \underline{\chi}_i = \sum_{j = N}^i (-1)^{i-j} r_j $$
    and as before, we define,
    $$ \underline{\fitt}^i_k(F^\kdot) = I_{\underline{\chi}_i-k}(d^i) $$
\end{definition}

\begin{corollary} \label{bounded-complex}
    Let
    \begin{center}
    \begin{tikzcd}
    0 \arrow[r] & F_0 \arrow[r, "d_0"] \arrow[d] & F_{1} \arrow[r] \arrow[d] & \dots \arrow[r] & F_n \arrow[r] \arrow[d] & 0 \\
    0 \arrow[r] & \widetilde{F}_0 \arrow[r, "\widetilde{d}_0"]          & \widetilde{F}_{1} \arrow[r]   & \dots \arrow[r] & \widetilde{F}_n \arrow[r]   & 0
    \end{tikzcd}
    \end{center}
    be a quasi-isomorphism of bounded cochain complexes of finite free $R$-modules. Then $\underline{\fitt}^i_k(F^\kdot) = \underline{\fitt}^i_k(\widetilde{F}^\kdot)$
\end{corollary}

\begin{proof}
    Let $\chi(F^\kdot) = \sum_{j = 0}^n (-1)^j r_j$. Then $\chi(F^\kdot) = \chi(\widetilde{F}^\kdot)$ since they are quasi-isomorphic. Now the corollary follows from applying Theorem \ref{Fitting-ideal},

    \[ \underline{\fitt}^i_k(F^\kdot) = \fitt^i_{(-1)^{i-1}\chi(F^\kdot)+ k}(F^\kdot) = \fitt^i_{(-1)^{i-1}\chi(\widetilde{F}^\kdot)+ k}(\widetilde{F}^\kdot) = \underline{\fitt}^i_k(\widetilde{F}^\kdot) \qedhere \]
\end{proof}

Note that $\underline{\fitt}^0_k(F^\kdot) = I_{r_0-k}(d_0)$ as this will be useful later when defining Brill-Noether ideals.

Thanks to Theorem \ref{Fitting-ideal}, we can define the Fitting ideals for any object of $D^-_{coh}(R)$, the bounded above derived category with coherent cohomology. Indeed, for any $K \in D^-_{coh}(R)$, we can choose a bounded below complex of finite free modules $E_\kdot \underset{qis}{\sim} K$ and define $\fitt_k^i(K) = \fitt_k^i(E_\kdot)$. Theorem \ref{Fitting-ideal} says that this is well-defined.

To show that we can glue fitting ideals together on a scheme, we need to show that they commute with localization. More generally, we show that they commute with flat base change.

\begin{lemma} \label{flat-base-change}
    Let $A_\kdot$ be a bounded above complex of $R$-modules and let $R \rightarrow S$ be a flat morphism. Then
    $$ \fitt_k^i(A_\kdot \otimes_R S) = \fitt_k^i(A_\kdot) \otimes_R S $$
\end{lemma}

\begin{proof}
    Let $F_\kdot$ be a bounded above complex of free $R$-modules that is quasi-isomorphic to $A_\kdot$. Since $R \rightarrow S$ is flat, $F_\kdot \otimes_R S$ is quasi-isomorphic to $A_\kdot \otimes_R S$. Therefore,
    \[ \fitt_k^i(A_\kdot \otimes_R S) = \fitt_k^i(F_\kdot \otimes_R S) = \fitt_k^i(F_\kdot) \otimes_R S = \fitt_k^i(A_\kdot) \otimes_R S \qedhere \]
\end{proof}

Using this, we can define the Fitting ideals for any object in $D_{coh}^-(X)$, the derived category of bounded above complexes with coherent cohomology on a locally noetherian scheme $X$. We can also define $\underline{\fitt}_k^i(K)$ for perfect objects in $D_{coh}^b(X)$ which are those that are locally quasi-isomorphic to a bounded complex of finite free modules. Now we prove the following generalization of \ref{flat-base-change},

\begin{proposition}
Let $X$ be a locally noetherian scheme, $K \in D_{coh}^-(X)$ and $g: X' \rightarrow X$ any morphism. Then
$$ g^{-1}\fitt_k^i(K) = \fitt_k^i(Lg^* K) $$
\end{proposition}
\begin{proof}
The statement is local so we assume $K$ is quasi-isomorphic to a complex of finite free sheaves $F$. Then $g^* F = Lg^* K$ and the lemma follows from Lemma \ref{minors-basechange}.
\end{proof}

Looking at the argument in Lemma \ref{fitt}, we see that if $M$ is a finite $R$-module (thought of as a complex concentrated in degree zero), then $\fitt_k^1(M) = \fitt_k(M)$. As a first application, we show how the higher Fitting ideals of $M$ measure the projective dimension of $M$. 

\begin{corollary} \label{projective-dimension}
    Let $R$ be a reduced noetherian ring and $M$ a finite $R$-module of rank $k$ (has dimension $k$ when localized at each minimal prime). Let $Z \subseteq \spec R$ be the closed subscheme defined by $\fitt_{(-1)^{d-1}k}^{d-1}(M)$. Then
    $$ \tr{Supp } Z = \{ \mathfrak{p} \in \spec R | \tr{ pd}(M_\mathfrak{p}) \geq d \} $$
\end{corollary}

\begin{proof}
    Construct a free resolution of $M$,
    \begin{center}
    \begin{tikzcd}
    \dots \arrow[r] & F_1 \arrow[r] & F_0 \arrow[r] & M \arrow[r] & 0
    \end{tikzcd}
    \end{center}

    Let $K_d$ be the kernel of the map $\phi: F_{d-1} \rightarrow F_{d-2}$. Then $\tr{pd}(M) \leq d$ if and only if $K_d$ is projective (see \cite[4.1.6]{Weibel_1994}). By localizing, we see that at each minimal prime, $K_d$ has rank $\chi_d + (-1)^d k$. Therefore, the points $\mathfrak{p} \in \spec R$ where $K_d$ is not projective are the points where $\phi$ has rank less than $r_{d-1} - (\chi_d + (-1)^d k) = \chi_{d-1} + (-1)^{d-1} k$ which is exactly the locus cut out by $I_{\chi_{d-1} + (-1)^{d-1}k}(\phi)$. Since $F_\kdot$ is a resolution, it is quasi-isomorphic to $M$, therefore $\fitt_{(-1)^{d-1}k}^{d-1}(M) = \fitt_{(-1)^{d-1}k}^{d-1}(F_{\cdot}) = I_{\chi_{d-1} + (-1)^{d-1}k}(\phi)$.
\end{proof}

\subsection{Application to Singular Varieties}

Much like classical Fitting ideals, we can use higher Fitting ideals to study the singular locus of a variety. Let $X$ be a reduced finite type scheme of pure dimension $d$.

For each $i$, let $\sing^i(X)$ be the closed subscheme defined by $\fitt_{(-1)^{i-1}d}^{i-1}(\Omega_X)$. By Corollary \ref{projective-dimension}, We know that set theoretically,
$$ \dots \subseteq \tr{Supp } \sing^1(X) \subseteq 
\tr{Supp } \sing^0(X) = \sing(X) $$
so it is natural to ask if this containment holds as subschemes.

\begin{question}
Let $X$ be a reduced finite type scheme of pure dimension $d$. Is $$\fitt_{(-1)^{i-1}d}^{i-1}(\Omega_X) \subseteq \fitt_{(-1)^{i}d}^{i}(\Omega_X)?$$
\end{question}

In fact, some examples have seemed to suggest that $\fitt_{(-1)^{i-1}d}^{i-1}(\Omega_X) \subseteq \fitt_{(-1)^{i}d-1}^{i}(\Omega_X)$. Note the extra $-1$ on the right hand side. This would have the pleasant consequence that 
\begin{equation} \label{sing-inc}
\sing^{i+1}(X) \subseteq \sing( \sing^i(X))
\end{equation}
by Proposition \ref{sing-determinantal-subschemes}.

\begin{example}
Let $X$ be the three coordinate axes in $\mathbb{A}^3$ defined by $xy=xz=yz=0$. We obtain a free resolution of $\Omega_X$,
$$ 0 \rightarrow \OO_X \overset{B}\rightarrow \OO_X^3 \overset{A}\rightarrow \OO_X^3 \rightarrow \Omega_X \rightarrow 0 $$
where 

\begin{center}
$A = \begin{bmatrix}
    y & x & z \\
    z & 0 & x \\
    0 & z & y
\end{bmatrix}$ \hspace{1.5cm} $B = \begin{bmatrix}
    x \\
    y \\
    z
\end{bmatrix}$
\end{center}

The scheme $\sing(X)$ is defined by the ideal generated by the $2$ by $2$ minors of $A$, therefore $\sing(X)$ is defined by 
$$ x^2 = y^2 = z^2 = xy = xz = yz $$

In other words, it is the fat point. The ideal $\fitt_{-1}^1(\Omega_X)$ is generated by the $1 \times 1$ minors of $B$, therefore $\sing^1(X)$ is the reduced point $x = y = z = 0$.
\end{example}

Next we show that these higher singular loci vanish for a locally complete intersection. This gives further evidence that the scheme structure on these higher singular loci reflects the geometry of the singularity.

\begin{proposition}
    Let $X$ be a reduced, finite type scheme. If $X$ is locally a complete intersection, then $\sing^i(X) = \varnothing$ for all $i > 0$.
\end{proposition}

\begin{proof}
    Embed $X$ into a smooth scheme $Y$. Since $X$ is locally a complete intersection, its conormal sheaf $\mathcal{I}/\mathcal{I}^2$ is locally free. Now consider the sequence,
    $$ \mathcal{I}/\mathcal{I}^2 \rightarrow \Omega_Y \otimes \OO_X \rightarrow \Omega_X \rightarrow 0 $$

    Since $X$ is reduced and $\mathcal{I}/\mathcal{I}^2$ is locally free, the first map is injective. We can use this free resolution to compute our higher fitting ideals and see that $\fitt_{(-1)^{i-1}d}^{i-1}(\Omega_X) = 0$ for $k > 1$.
\end{proof}

This proposition along with \ref{sing-inc} lead to the following question,

\begin{question}
    Let $X$ be a quasi-projective variety. If $\sing(X)$ (the scheme theoretic singular locus) is smooth, does this imply that $X$ is a local complete intersection variety?
\end{question}

\section{Brill-Noether Ideals} \label{Brill-Noether Ideals}

We now turn our attention to constructing generalizations of the classical Brill-Noether spaces. Recall that given a smooth curve $C$, the Brill-Noether space $W^r_d$ consists of line bundles of degree $d$ whose space of global sections has dimension at least $r+1$. The space $W^r_d$ parameterizes families of degree $d$ line bundles $L$, that satisfy the homological condition that $\fitt_{g-d+r} R^1\pi_*L = 0$. Written more precisely, we have that $W^r_d$ represents the functor,
$$ Y \mapsto \left\{ L \in \Pic^d_{X/S}(Y) : \fitt_{g-d+r}(R^1 \pi_* L) = 0 \right\} $$

We wish to generalize this to an arbitrary scheme. To do this requires two steps, first we need to remove the dependence on degree since this is only a good notion for curves. We do this by considering the locus $W^k_X$ of line bundles whose space of global sections has dimension at least $k+1$ all at once rather than breaking it up by component of $\pic_X$. This will give us a scheme that is not noetherian since it will live on infinitely many components of $\pic_X$ but this is not an issue as the connected components are noetherian. The second step is to find the right generalization of the homological condition $\fitt_{g-d+r}(R^1 \pi_* L) = 0$. This is made more difficult by the fact that we no longer have a notion of degree and we have more than two cohomology groups to deal with, so we must generalize this a different way. This is where we use our higher Fitting ideals.

\begin{definition}
    Let $f: X \rightarrow Y$ be a proper morphism of locally noetherian schemes and let $\F$ be a coherent sheaf on $X$, flat over $Y$. We define the \textbf{Brill-Noether ideals} of $\F$ as,
    $$ \BN^k(\F) = \underline{\fitt}_k^0(Rf_* \F) $$

    We define the \textbf{Brill-Noether loci} of $\F$ as the closed subschemes determined by $\BN^k(\F)$ and denote it $W^k(\F)$.
\end{definition}

We suppress the dependence on $f$ as the morphism is usually clear from context. We will denote them $\BN^k_f(\F)$ and $W^k_f(\F)$ if the morphism is not clear. Now we show that these Brill-Noether loci are supported on the same set as in the case of curves. 

\begin{theorem} \label{BN-ideals}
    Let $f: X \rightarrow Y$ be a proper morphism of locally noetherian schemes and let $\F$ be a coherent sheaf on $X$, flat over $Y$. Then
    $$ \tr{Supp } W^k(\F) = \{y \in Y | h^0(X_y, \F_y) \geq k + 1\} $$

\end{theorem}

\begin{proof}

This is local so we may assume $Y$ is affine. By \cite[5]{Mumford70}, we may choose a bounded complex of locally free sheaves on $Y$,

\begin{equation} \label{base-change-complex}
 \mathcal{E}^0 \overset{\delta^0}{\rightarrow} \mathcal{E}^1 \overset{\delta^1}{\rightarrow} \dots \overset{\delta^{n-1}}{\rightarrow} \mathcal{E}^n 
\end{equation}
that is quasi-isomorphic to $Rf_* \F$. By cohomology and base change (see \cite[\href{https://stacks.math.columbia.edu/tag/071M}{Tag 071M}]{stacks-project2}) $\mathcal{E}^\kdot$ has the following property:

For any morphism $T \rightarrow Y$,
\begin{equation} \label{bc-complex}
    R^i(f_T)_* \F_T \cong h^i(\mathcal{E}_T^\kdot)
\end{equation}

In particular, for any point $y \in Y$, we have that

\begin{equation} \label{base-change-complex-eq}
 H^0(X_y,\F_y) = \ker \delta^0_y 
\end{equation}

Let $r_0 = \tr{rank } \mathcal{E}^0$. By definition, $\BN^k(\F) = I_{r_0-k}(\delta^0)$. Therefore, $W^k(\F)$ is the locus where the rank of $\delta^0$ is at strictly less than $r_0-k$. By \ref{base-change-complex-eq}, this is the locus where $h^0(X_y,\F_y) \geq k+1$.
\end{proof}

It follows from the property \ref{bc-complex} that Brill-Noether ideals commute with arbitrary base change. Note the difference from higher fitting ideals that only commute with flat base change.

\begin{corollary} \label{BN-basechange}
Let $f: X \rightarrow Y$ be a proper morphism of locally noetherian schemes and let $\F$ be a coherent sheaf on $X$, flat over $Y$. Let $g: T \rightarrow X$ be any morphism. Then
$$ g^{-1} \BN^k(\F) = \BN^k(g^* \F) $$
\end{corollary}

As a sanity check, we calculate $\BN^k(\F)$ in the simplest possible case.

\begin{example}
Let $f: X \rightarrow Y$ be a proper morphism of locally noetherian schemes and let $\F$ be a coherent sheaf on $X$, flat over $Y$. Assume $Y$ is connected. Suppose that $H^i(X_y,\F_y) = 0$ for all $i > 0$ and all $y \in Y$. Since $Y$ is connected and $\F$ is flat over $Y$, the Euler characteristic of $\F$ is constant which implies that $h^0(X_y,\F_y)$ is the same for all $y \in Y$ (call this value $N$). Cohomology and base change \cite[III.12.11]{Hartshorne77} implies that $Rf_* \F = f_*\F$ and $f_* \F$ is locally free of rank $N$. Now we can easily calculate that

  \begin{equation*}
    \BN^k(\F) = \underline{\fitt}_k^0(f_* \F) =
    \begin{cases*}
      \OO_Y & if $k < N$ \\
      0        & if $k \geq N$
    \end{cases*}
  \end{equation*}
  
\end{example}

\subsection{Families of Sheaves with Liftable Sections} \label{liftable-sections}
Let $f: X \rightarrow Y$ be a proper morphism of locally noetherian schemes and let $\F$ be a coherent sheaf on $X$, flat over $Y$. Given an open subset $U \subseteq Y$ and a section $s \in H^0(X, f^{-1}(U))$, we may restrict $s$ to a section of $H^0(X_y,\F_y)$ for any $y \in U$. In this section, we are interested in when a section of $H^0(X_y,\F_y)$ can be lifted to a section on $f^{-1}(U)$ for some $U$.

\begin{definition}
    With notation as above, we say that $\F$ has \textbf{liftable sections over $Y$} if the natural map
    $$ f_* \F \otimes k(y) \rightarrow H^0(X_y,\F_y) $$
    is surjective for all $y$.
\end{definition}

We show how the Brill-Noether ideals characterize the property of having liftable sections.

\begin{proposition} \label{grauert}
    Let $f: X \rightarrow Y$ be a proper morphism of locally noetherian schemes and let $\F$ be a coherent sheaf on $X$, flat over $Y$. Assume that $Y$ is connected. The following are equivalent:
    \begin{enumerate}
   \item $\F$ has liftable sections over $Y$.
   \item There is some $k$ such that $\BN^{k}(\F) = 0$ and $\BN^{k+1}(\F) = \OO_Y$.
  \end{enumerate}

  Furthermore, if these conditions hold then $f_* \F$ is locally free of rank $k+1$ and $\phi_y$ is an isomorphism for all $y$.
\end{proposition}

    \begin{proof}
        We work locally on $Y$ so we may choose a complex $\mathcal{E}^\kdot$ of finite locally free sheaves on $Y$ that is quasi-isomorphic to $Rf_* \F$ as in \ref{base-change-complex}. Let $\mathcal{W}$ be the cokernel of the map $\mathcal{E}^0 \rightarrow \mathcal{E}^1$,
        $$ \mathcal{E}^0 \rightarrow \mathcal{E}^1 \rightarrow \mathcal{W} \rightarrow 0 $$
 
        It follows from \cite[III, 12.4]{Hartshorne77} and \cite[III, 12.5]{Hartshorne77} that $\phi_y$ is surjective for all $y$ if and only if $\mathcal{W}$ is locally free. The Brill-Noether ideals are exactly the fitting ideals of $\mathcal{W}$, so the equivalence of (1) and (2) follows from Proposition \ref{locally-free-fitting}.

        Now assume the conditions hold. Then (1) implies, using cohomology and base change \cite[III, 12.11]{Hartshorne77}, that $\phi_y$ is an isomorphism for all $y$ and $f_* \F$ is locally free of rank $k$.
        \end{proof}

This proposition, along with Corollary \ref{BN-basechange}, implies that if $T$ is connected and we have a morphism $T \rightarrow Y$ such that $\F_{X_T}$ has liftable sections over $T$, then $T \rightarrow Y$ must factor through $W^k(\F) \setminus W^{k+1}(\F)$ for some $k$. Putting this all together, we obtain:

\begin{theorem} \label{functorial}
    Let $f: X \rightarrow Y$ be a proper morphism of locally noetherian schemes and let $\F$ be a coherent sheaf on $X$, flat over $Y$. Then
    $$ \dot{\bigcup}_{k \in \mathbb{Z}} W^k(\F) \setminus W^{k+1}(\F) \tr{ represents the functor } $$
$$ T \rightarrow  \left\{ T \rightarrow Y : \F_T \tr{ has liftable sections over } T \right\} $$
\end{theorem}

This along with Theorem \ref{BN-ideals} allow us to construct Brill-Noether spaces for any moduli space of sheaves.

\begin{corollary} \label{BN-moduli-of-sheaves}
Let $X$ be a projective scheme and let $M$ be a moduli space of sheaves on $X$ with universal sheaf $\mathcal{U}$. Then
$$ \tr{Supp } W^k(\mathcal{U}) = \{\F \in M | h^0(X, \F) \geq k + 1\} $$

and
$$ \dot{\bigcup}_{k \in \mathbb{Z}} W^k(\mathcal{U}) \setminus W^{k+1}(\mathcal{U}) \tr{ represents the functor } $$
$$ T \rightarrow  \left\{ T \rightarrow M : \mathcal{U}_T \tr{ has liftable sections over } T \right\} $$

\end{corollary}

Note that even if we do not have a universal sheaf (e.g. $\pic_X$ when $X$ has no rational points), we can still obtain this result by descent since Brill-Noether ideals commute with base change. We choose to omit the details of this for the sake of simplicity.

\subsection{Comparison to the Classical Case}

Our main application of Theorem \ref{BN-ideals} is to moduli spaces of sheaves on a fixed variety. Let $M$ be a moduli space of sheaves on a variety $X$ with universal sheaf $\mathcal{U}$ on $M \times X$. Consider the space $W^k(\mathcal{U})$. By Theorem \ref{BN-ideals},
$$ \tr{Supp } W^k(\mathcal{U}) = \{ \F \in M | h^0(X,\F) \geq k+1\} $$

In the case where $X$ is a smooth curve and $M = \pic^d_C$, the classical Brill-Noether loci, $W^k_d$, are defined by the ideal $\fitt_{k-\chi}(R^1\pi_* \mathcal{U})$ where $\chi = d-g+1$. By Riemann-Roch, $\chi$ is equal to the Euler characteristic of $\LL$ for any $\LL \in \pic^d_C$. Our next result implies that $W^k(\mathcal{U}) = W^k_d$ in this case.

More generally, in \cite{Costa10}, the authors define Brill-Noether loci for moduli of stable vector bundles on an arbitrary variety with the additional condition that $H^i(X,E) = 0$ for $i \geq 2$ for all vector bundles $E$ in the moduli space. They also define them with the ideals $\fitt_{k-\chi}(\pi_* \mathcal{U})$ where $\chi$ is the Euler characteristic of any $E$ in the moduli space. The following result shows that $W^k(\mathcal{U})$ agrees with the Brill-Noether loci defined in \cite{Costa10} as well.

\begin{corollary} \label{classical}
    Let $f: X \rightarrow Y$ be a proper morphism of locally noetherian schemes, $\F$ a coherent sheaf on $X$, flat over $Y$ and assume $H^i(X_y,\F_y) = 0$ for all $i \geq 2$ and all $y \in Y$. Assume further that $\chi(\F_y)$ is the same for all $y$, call this value $\chi$. Then $\mathcal{BN}^k(\F) = \fitt_{k-\chi}(R^1f_* \F)$.
\end{corollary}

\begin{proof}
    It follows from our assumption that $R^if_*\F = 0$ for $i \geq 2$. Then it follows from the proof of Lemma 1 in \cite[5]{Mumford70} that we can choose our complex $\mathcal{E}^\kdot$ (that we use to define $\mathcal{BN}^k(\F)$) to be 0 for $i \geq 2$. So we end up with a complex with only two nonzero terms,

    \begin{center}

    \begin{tikzcd}
    0 \arrow[r] & \mathcal{E}^0 \arrow[r, "\delta^0"] & \mathcal{E}^1 \arrow[r] & 0
    \end{tikzcd}
    \end{center}
    therefore the cokernel of $\delta^0$ is $R^1f_* \F$. So 
    
    \[ \fitt_{k-\chi}(R^1f_* \F) = I_{r_1-(k - \chi)}(\delta^0) = I_{r_0-k}(\delta^0) = \mathcal{BN}^k(\F) \qedhere \]
\end{proof}

\section{Brill-Noether Loci in the Picard Scheme} \label{Brill-Noether Loci-chapter}

For a smooth curve $C$ of genus $g$, the classical Brill-Noether variety $W_d^k$ has support
$$ \tr{Supp } W_d^k = \{ \LL \in \pic_C^d : h^0(C,\LL) \geq k+1\} $$
and represents the functor
$$ Y \rightarrow \left\{ L \in \Pic_C^d(Y) : \fitt_{k - \chi}(R^1\pi_*L) = 0 \right\} $$
where $\pi: C \times Y \rightarrow Y$ is the projection and $\chi = g-d+1$. See \cite[IV.3]{ACGH1} for details.

Our next result is a direct generalization of this for higher dimensional varieties where the fitting ideal is replaced with the Brill-Noether ideal.

\begin{corollary} \label{Brill-Picard}
    Let $X$ be a projective variety. Let $\pi: X \times \pic_X \rightarrow \pic_X$ be the projection and let $\mathcal{U}$ be a universal line bundle on $X \times \pic_X$. Then
$$ \tr{Supp } W^k(\mathcal{U}) = \{\LL \in \pic_X : h^0(X,\LL) \geq k+1\} $$
    and $$ \dot{\bigcup}_{k \in \mathbb{Z}} W^k(\mathcal{U}) \setminus W^{k+1}(\mathcal{U}) \tr{ represents the functor } $$
$$ T \rightarrow  \left\{ L \in \Pic_X(T) : L \tr{ has liftable sections over } T \right\} $$
\end{corollary}

\begin{proof}
    This follows directly from Theorem \ref{BN-moduli-of-sheaves} applied to $\pic_X$ and $\mathcal{U}$.
\end{proof}

Because of this, we define $W_X^k$ to be $W^k(\mathcal{U})$ from this corollary.

\begin{remark}
    Corollary $\ref{classical}$ shows that $W_C^k = \bigcup_d W_d^k$ in the case $X$ is a smooth curve. 
\end{remark}

There is also a relative version of Brill-Noether varieties. Let $p: C \rightarrow S$ be a projective smooth family of curves that admits a section. Then there exist relative Brill-Noether varieties $\mathcal{W}_d^k(p) \rightarrow Y$ such that the fiber of $\mathcal{W}_d^k(p)$ at $s$ is $W_d^k$ for the curve $C_s$. Furthermore, $\mathcal{W}_d^k$ represents the functor,
$$ Y \rightarrow \left\{ L \in \Pic_{C/S}^d(Y) : \fitt_{k-\chi}(R^1p_{Y*}L) = 0 \right\} $$

See \cite{ACGH2} for details.

Now let $f: X \rightarrow S$ be a projective morphism of noetherian schemes that admits a section. Then by \cite{kleiman2005picard}, the relative Picard scheme $\pic_{X/S}$ exists and there is a universal line bundle $\mathcal{U}$ on $X \times_S \pic_{X/S}$. Define $\mathcal{W}^k(f) = W^k(\mathcal{U})$. Then by Theorem \ref{BN-ideals}, we obtain

\begin{corollary}
    The fiber of $\mathcal{W}^k(f)$ at $s$ is $W_{X_s}^k$ and it represents the functor
$$ Y \rightarrow \left\{ L \in \Pic_{X/S}(Y) : \BN^k(L) = 0 \right\} $$
    
\end{corollary}

\subsection{Deforming Embeddings} \label{deforming-embeddings}

Let $\LL$ be a line bundle on a scheme $X$. Then $\LL$ naturally defines a rational map,
$$ \phi_\LL: X \dashrightarrow \mathbb{P}(H^0(X,\LL)) $$

Given a family of line bundles $L \in \Pic_X(T)$, it is natural to ask if the maps $\phi_{L_t}$ form a family as well. It turns out that this happens if and only if $L$ has liftable sections over $T$. Assume that $L$ has liftable sections over $T$. Then $\pi_{T*} L$ is locally free by Proposition \ref{grauert}. The adjunction map $\pi_{T}^*\pi_{T*} L \rightarrow L$ induces a natural rational map 
$$ \Phi: X \times T \dashrightarrow \mathbb{P}(\pi_{T*} L) $$

Given a point $t \in T$, we see that $\Phi_{X_t}: X_t \dashrightarrow \mathbb{P}^N$ is the rational map to projective space determined by the linear system
$$ \tr{Im}[\pi_{T*} L \otimes k(t) \rightarrow H^0(X_t,L_t)] $$

Since $L$ has liftable sections, $\Phi_{X_t} = \phi_{L_t}$. It is clear that if $L$ does not have liftable sections, then there is no way to form a family with the $\phi_{L_t}$ in this way.

Being very ample or basepoint-free depends entirely on the map $\phi_\LL$ so it makes sense to consider families with liftable sections when considering families of very ample or basepoint-free line bundles.

\section{Moduli of Very Ample Line Bundles} \label{Moduli of VA-chapter}

In this section, we will construct a moduli space for very ample line bundles on a variety. A natural place to start is by looking at the set of points of $\pic_X$ that correspond to very ample line bundles and see if we can give it the structure of a subscheme. It turns out this cannot happen, in section \ref{examples} we show that the locus of very ample line bundles need not be locally closed even in the case where $X$ is a smooth curve. The next best thing we could hope for is that the locus of very ample line bundles is a union of finitely many subschemes, with each one naturally parameterizing families of very ample line bundles with some extra property. This is what we do in Theorem \ref{va-moduli} using Theorem \ref{Brill-Picard}.

First, let us define the very ample locus of a family of line bundles. Let $f: X \rightarrow Y$ be a morphism and $L$ a line bundle on $X$. We define,
$$ \va(f,L) = \{y \in Y : L_y \tr{ is very ample}\} $$

Let $\mathcal{U}$ be a universal line bundle on $X \times \pic_X$ and $\pi$ be the projection $X \times \pic_X \rightarrow \pic_X$. We define $\va(X) = \va(\pi, \mathcal{U})$. We see that the closed points of $\va(X)$ are exactly the locus of line bundles $\LL \in \pic_X$ such that $\LL$ is very ample. If $X$ is projective then it follows from Corollary \ref{constructible} that $\va(X)$ is constructible. Therefore, if $\va(X)$ is not locally closed, then there is no subscheme of $\pic_X$ whose closed points correspond to the locus of very ample line bundles in $\pic_X$.

Similarly, we define the basepoint free locus of $f$ as,
$$ \bpf(f,L) = \{y \in Y : L_y \tr{ is basepoint free} \} $$

\subsection{Counterexamples} \label{examples}

In this section we show that for a non-hyperelliptic smooth curve $C$ of genus at least 5, $\va(C)$ is not locally closed. We show that $\va(C)$ is not locally closed in any neighborhood of the canonical bundle by showing that a general deformation of $K_C$ is very ample then writing an explicit deformation of $K_C$ that is not very ample. The picture to have in mind for what $\va(C)$ looks like is the plane with the $x$ and $y$ axes removed but the origin added back in, where the origin corresponds to the canonical bundle and the removed axes correspond to the special deformations.

\begin{theorem}
    Let $C$ be a non-hyperelliptic curve of genus at least 5. Then $\va(C)$ is not locally closed.
\end{theorem}

\begin{proof}
    We will show that $\pic_C^{2g-2} \cap \va(C)$ is not locally closed. Since $2g-2 \ge d+3$, it follows from \cite[IV.6.1]{Hartshorne77} that a general member of $\pic_C^{2g-2}$ is very ample. Thus to prove the claim, it suffices to show $\pic_C^{2g-2} \cap \va(C)$ is not open. To do this, we write down a specific family of degree $2g-2$ line bundles whose special member is very ample but whose general member is not.

    Let $P$ be a point of $C$. Consider the family of line bundles $L$ on $C \times C$ whose fiber over a point $t \in C$ is $\OO(K_C + P - t)$. At $t = P$, we get $\OO(K_C)$ which is very ample since $C$ is non-hyperelliptic. We will show that for $t \ne P$, $\OO(K_C + P - t)$ is not very ample (in fact it is not even basepoint free) which proves our claim. 

    So assume $t \ne P$. By \cite[IV.3.1]{Hartshorne77} and Riemann-Roch, $\OO(K_C + P - t)$ is basepoint free if and only if $h^0(C, \OO(t - P + Q)) = h^0(C, \OO(t - P))$ for all $Q \in C$. For $Q = P$, we have $h^0(C, \OO(t - P + Q)) = h^0(C, \OO(t)) = 1$ while $h^0(C , \OO(t - P)) = 0$, thus it is not basepoint free.
\end{proof}

Using the same idea, we can give a very explicit family on line bundles whose very ample locus is not locally closed.

\begin{example}
    Let $C$ be a smooth curve of gonality at least 6 and let $P$ be a point of $C$. Consider the family of line bundles $L$ on $C$ over $C \times C$ whose fiber at $(x,y)$ is

    $$ L_{(x,y)} = \OO(K_C + 3P - 2x - y) $$

    We leave it as an exercise to check that $L_{(x,y)}$ is very ample if and only if $p \notin \{x,y\}$ or $(x,y) = (p,p)$. In other words, if both $x$ and $y$ are both not $p$, than $L_{(x,y)}$ is very ample and the only other very ample member of the family is at $(p,p)$. In this case, the locus $\va(\pi_{C \times C}, L)$ is obtained by removing the curves $x = p$ and $y = p$ but putting the point $(p,p)$ back in, which is clearly not locally closed.
\end{example}

\subsection{Very Ampleness in Families} \label{main}

As we showed in the previous section, very ampleness is not an open property even in proper flat families. In this section, we prove that very ampleness is an open property if we impose additional restraints on the family. In particular, we prove that very ampleness is open if the dimension of the space of global sections is constant on the fibers or if the first cohomology groups of the line bundles in the family vanish. This will be enough for constructing our moduli space in Theorem \ref{va-moduli}.

The proof we give is motivated by the proof of Proposition B.2.10 in the the unpublished notes on Stacks and Moduli by Jarod Alper which can be found at \newline \href{https://sites.math.washington.edu/~jarod/}{https://sites.math.washington.edu/$\sim$jarod/}.

\begin{lemma} \label{bf}
    Let $f: X \rightarrow Y$ be a proper flat morphism of locally noetherian schemes and let $L$ be a line bundle on $X$ such that $h^0(X_y,L_y)$ is constant for all $y \in Y$. Then $\bpf(f,L)$ is open.
\end{lemma}

\begin{proof}
    We can start by base changing to make $Y$ reduced since that will not change any of the fibers. Then, by Grauert's Theorem, $f_*L$ is locally free and the natural map
    $$ f_*L \otimes k(y) \rightarrow H^0(X_y,L_y) $$
    is an isomorphism for all $y \in Y$. Let $y_0 \in Y$ be such that $L_{y_0}$ is basepoint free. We can shrink $Y$ so that $f_* L$ is free with generators, $s_1,\dots,s_n \in H^0(X,L)$. These sections restrict to a basis of $H^0(X_y,L_y)$ for each $y \in Y$. Let $Z$ be the intersection of the zero sets of the $s_i$. Then $f(Z)$ is closed and its complement is exactly the locus of $y \in Y$ such that $L_y$ is basepoint free.
\end{proof}

\begin{theorem} \label{const}
    Let $f: X \rightarrow Y$ be a proper flat morphism of locally noetherian schemes and let $L$ be a line bundle on $X$ such that $h^0(X_y,L_y)$ is constant for all $y \in Y$. Then $\va(f,L)$ is open. In particular if $L$ has liftable sections then $\va(f,L)$ is open.
\end{theorem}
\begin{proof}
    Let $y_0 \in Y$ be such that $L_{y_0}$ is very ample. As in Lemma \ref{bf}, we pick a basis $s_0,\dots,s_n$ for $f_*(L)$ that restricts to a basis for each $H^0(X_y,L_y)$. We can shrink $Y$ so that $Y$ is affine, noetherian and the $s_i$ generate $L$ by Lemma \ref{bf}. Now we can look at the morphism determined by the $s_i$, $\phi:X \rightarrow \PP^n_Y$. Note that on a fiber of $f$, $\phi$ restricts to the morphism determined by a basis of $H^0(X_y,L_y)$, thus $L_y$ is very ample if and only if $\phi$ restricts to a closed immersion on $X_y$.
    
    By semicontinuity of fiber dimension, we may shrink $Y$ further and assume $\phi$ is quasi-finite which implies that $\phi$ is finite since it is proper. The map on structure sheaves, $\OO_{\PP^n_Y} \rightarrow \phi_*(\OO_X)$ has coherent cokernel, which is zero along $X_{y_0}$. Thus we can shrink $Y$ so that $\OO_{\PP^n_Y} \rightarrow \phi_*(\OO_X)$ is surjective, making $\phi$ a closed immersion. This means that for all $y$ in this neighborhood, the morphism determined by a basis of $H^0(X_y,L_y)$ is a closed immersion, which is exactly what it means to be very ample.
\end{proof}

\begin{corollary} \label{constructible}
    Let $f: X \rightarrow Y$ be a proper morphism of locally noetherian schemes and let $L$ be a line bundle on $X$. Then $\bpf(f,L)$ and $\va(f,L)$ are constructible. 
\end{corollary}
\begin{proof}
    Like in the previous proofs, we may assume $Y$ is reduced. We may assume $f$ is flat by generic flatness. We may assume $h^0(X_y,L_y)$ is constant by semicontinuity. Now the corollary follows from \ref{bf} and \ref{const}.
\end{proof}

The proof of \ref{const} can be easily modified to prove the same result in the case where $h^1(X_y,L_y) = 0$ but we don't assume $h^0(X_y,L_y)$ is constant. We just replace Grauert's Theorem with cohomology and base change. We provide a different proof, which came from a conversation with Nathan Chen. This proof highlights how the first cohomology group is acting as an obstruction. Furthermore, this proof easily generalizes to the $k$-very ample case so we do that as well.

A line bundle $\LL$ on a projective scheme $X$ is called $k$-very ample if for every 0-dimensional subscheme $\xi$ such that $h^0(\xi,\OO_\xi) = k+1$, the restriction map $H^0(X,\LL) \rightarrow H^0(\xi,\LL_\xi)$ is surjective. Note that 0-very ample is equivalent to basepoint free and 1-very ample is equivalent to very ample. Let $Hilb^{k+1}(X)$ be the Hilbert scheme of 0-dimensional subschemes of $X$ of length $k+1$.

\begin{lemma} \label{k-very-ample}
    Let $X$ be a projective scheme and $\LL$ a line bundle on $X$. Assume that $H^1(X,\LL) = 0$. Then $\LL$ is $k$-very ample if and only if $H^1(X,\LL \otimes \mathcal{I}_\xi) = 0$ for every $\xi \in Hilb^{k+1}(X)$.
\end{lemma}

\begin{proof}
    Consider the short exact sequence,
    $$ 0 \rightarrow \mathcal{I}_\xi \otimes \LL \rightarrow \LL \rightarrow \LL_\xi \rightarrow 0 $$

    The lemma follows from the long exact sequence we get from taking cohomology,

    \[ H^0(X,\LL) \rightarrow H^0(\xi,\LL_\xi) \rightarrow H^1(X,\LL \otimes \mathcal{I}_\xi) \rightarrow 0 \qedhere \]
    
\end{proof}

\begin{proposition}
    Let $f: X \rightarrow Y$ be a proper flat morphism of locally noetherian schemes and let $L$ be a line bundle on $X$. Suppose that $L_{y_0}$ is $k$-very ample and $H^1(X_{y_0},L_{y_0}) = 0$ for some $y_0 \in Y$. Then there is an open neighborhood $U$ of $y_0$ such that $L_y$ is $k$-very ample for all $y \in U$.
\end{proposition}

\begin{proof}
    By semicontinuity, we may assume that $h^1(X_y,L_y) = 0$ for all $y \in Y$.
    Let $Hilb^{k+1}(X/Y)$ be the relative Hilbert scheme parameterizing 0-dimensional, length $k+1$ subschemes contained in the fibers of $f$. So the points of $Hilb^{k+1}(X/Y)$ are of the form $(y,\xi)$ where $y \in Y$ and $\xi \in Hilb^{k+1}(X_y)$. Then we have a fiber square,

    \begin{center}
        \begin{tikzcd}
X \underset{Y}{\times} Hilb^{k+1}(X/Y) \arrow[r] \arrow[d,"f'"] & X \arrow[d,"f"] \\
Hilb^{k+1}(X/Y) \arrow[r]                                  & Y          
\end{tikzcd}
    \end{center}

Let $Z = X \times_Y Hilb^{k+1}(X/Y)$. Let $\mathcal{U}$ be the universal subscheme in $Z$ and $\mathcal{I}_{\mathcal{U}}$ its ideal sheaf. The sheaf $L_Z \otimes \mathcal{I}_{\mathcal{U}}$ is flat over $Hilb^{k+1}(X/Y)$ so we can use semicontinuity with respect to $f'$ to see that the points $(y,\xi) \in Hilb^{k+1}(X/Y)$ such that $H^1(X_y, L_y \otimes \mathcal{I}_\xi) \neq 0$ is closed. The image of this closed set in $Y$ is closed and its complement is the open set we desired by Lemma \ref{k-very-ample}. 
\end{proof}

\subsection{The Moduli Functor}

Let $X$ be a projective variety. We define the moduli functor for very ample line bundles on $X$ as the functor $\va_X : Sch \rightarrow Set$ defined by
$$ \va_X(T) = \left\{ L \in \Pic_X(T) : \forall t \in T, L_t \tr{ is very ample and } L \tr{ has liftable sections over } T \right\} $$

Recall that $\Pic_X(T) = \Pic(X \times T)/ \pi_T^* \Pic(T)$, so $L \in \Pic_X(T)$ is a family of line bundles on $X$. Note that $\va_X$ is a subfunctor of $\pic_X$ so there is a natural map of functors $\va_X \rightarrow \pic_X$.

In Section \ref{deforming-embeddings}, we showed that a family has liftable sections over $T$ if and only if the maps $\phi_{L_t}$ deform in a family over $T$. When considering a family of very ample line bundles, it makes sense to require it to have liftable sections since being very ample is a property dependent on the map $\phi_{L_t}$.
Note that closed points of $\va_X$ are the (isomorphism classes of) very ample line bundles on $X$.

\subsection{Constructing the Moduli Space}

Theorem \ref{const} allows us to use the Brill-Noether loci of $\pic_X$ to prove that $\va_X$ is representable. 

\begin{theorem} \label{va-moduli}
        Let $X$ be a projective variety. The functor $\va_X$ is representable by a scheme. Moreover, the natural map
$$ \va_X \rightarrow \pic_X $$

is a locally closed embedding on the connected components of $\va_X$.

\end{theorem}

\begin{proof}
    Let $\mathcal{U}$ be a universal line bundle on $\pic_X \times X$. By Corollary \ref{Brill-Picard}, \newline $\dot{\bigcup}_{k \in \mathbb{Z}} W^k(\mathcal{U}) \setminus W^{k+1}(\mathcal{U}) \tr{ represents the functor } $
 $$ T \rightarrow  \left\{ L \in \Pic_X(T) : L \tr{ has liftable sections over } T \right\} $$

 Now it follows from Theorem \ref{const} that $\va_X$ is represented by an open subscheme of $\dot{\bigcup}_{k \in \mathbb{Z}} W^k(\mathcal{U}) \setminus W^{k+1}(\mathcal{U})$. Since the natural map 
 $$ \dot{\bigcup}_{k \in \mathbb{Z}} W^k(\mathcal{U}) \setminus W^{k+1}(\mathcal{U}) \rightarrow \pic_X $$

 is a locally closed immersion on connected components, the theorem follows.
\end{proof}

\begin{corollary}
    Let $C$ be a general smooth curve of genus at least 2. Then $\va_C$ is smooth.
\end{corollary}

\begin{proof}
    This follows from the fact that for a general smooth curve of genus at least 2, $W_d^k \setminus W_d^{k+1}$ is smooth for all $d$ and $k$, see \cite[XXI]{ACGH2}.
\end{proof}

\section{Moduli of Basepoint Free Line Bundles} \label{moduli of BF-chapter}
    We can use the exact same method to construct a moduli space for basepoint free line bundles. We define our moduli functor
    $$ \bpf_X(T) = \left\{ L \in \Pic_X(T) : \forall t \in T, L_t \tr{ is basepoint free and } L \tr{ has liftable sections over} T \right\} $$

    Lemma \ref{bf} proves that $\bpf_X$ is an open subscheme of $\dot{\bigcup}_{k \in \mathbb{Z}} W^k(\mathcal{U}) \setminus W^{k+1}(\mathcal{U})$, therefore we have,

    \begin{theorem}
        Let $X$ be a projective variety. The functor $\bpf_X$ is representable by a scheme. Moreover, the natural map
$$ \bpf_X \rightarrow \pic_X $$

is a locally closed embedding on the connected components of $\va_X$.
    \end{theorem}

\nocite{KM98}
\nocite{Kollar96}
\nocite{Hartshorne77}
\nocite{SingBook}
\nocite{MR1322960}
\nocite{Matsumura80}
\nocite{Matsumura89}
\nocite{Weibel_1994}
\nocite{Iitaka82}
\nocite{ACGH2}
\nocite{PosI}
\nocite{PosII}
\nocite{Mumford70}
\nocite{Huybrechts_Lehn_2010}
\nocite{MR1406314}
\nocite{MR2289519}

\bibliographystyle{alpha} % We choose the "plain" reference style
\bibliography{Ref,sample}

\end{document}